\documentclass{tsp-short}

\usepackage{stmaryrd}
\usepackage{bbm}
\usepackage{url}

\newcommand{\me}{\mathbb{E}}
\newcommand{\mz}{\mathbb{Z}}
\newcommand{\mr}{\mathbb{R}}
\newcommand{\mn}{\mathbb{N}}
\newcommand{\mmp}{\mathbb{P}}

\newcommand{\od}{\overset{\mathrm{d}}{=}}
\newcommand{\dod}{\overset{\mathrm{d}}{\to}}
\DeclareMathOperator{\1}{\mathbbm{1}}

\newtheorem{thm}{Theorem}[section]
\newtheorem{lem}{Lemma}[section]

\newtheorem{prp}{Proposition}[section]
\theoremstyle{remark}
\newtheorem{rem}{Remark}[section]
\theoremstyle{definition}

\begin{document}

\title
[Random processes with immigration]
{A note on convergence to stationarity of random processes with immigration}

\author{A.~V.~Marynych}
\address{Faculty of Cybernetics, Taras Shevchenko National University of Kyiv, 01601 Kyiv, Uk\-ra\-ine}
\curraddr{Institut f\"{u}r Mathematische Statistik, Westf\"{a}lische Wilhelms-Universit\"{a}t M\"{u}nster, 48149 M\"{u}nster, Germany}
\email{marynych@unicyb.kiev.ua}
\thanks{The work of the author was supported by the Alexander von Humboldt Foundation.}



\begin{abstract}
Let $X_1, X_2,\ldots$ be random elements of the Skorokhod space
$D(\mr)$ and $\xi_1, \xi_2, \ldots$ positive random variables such
that the pairs $(X_1,\xi_1), (X_2,\xi_2),\ldots$ are independent
and identically distributed. The random process $Y(t):=\sum_{k
\geq 0}X_{k+1}(t-\xi_1-\ldots-\xi_k)\1_{\{\xi_1+\ldots+\xi_k\leq
t\}}$, $t\in\mr$, is called random process with immigration at the
epochs of a renewal process. 
Assuming that the distribution of $\xi_1$ is nonlattice and has
finite mean while the process $X_1$ decays sufficiently fast, we
prove weak convergence of $(Y(u+t))_{u\in\mr}$ as $t\to\infty$ on
$D(\mr)$ endowed with the $J_1$-topology. The present paper
continues the line of research initiated in
\cite{Iksanov+Marynych+Meiners:2015-1,Iksanov+Marynych+Meiners:2015-2}.
Unlike the corresponding result in
\cite{Iksanov+Marynych+Meiners:2015-2} arbitrary dependence
between $X_1$ and $\xi_1$ is allowed.
\end{abstract}

\maketitle

\section{Introduction}

Denote by $D(\mr)$ the Skorokhod space of right-continuous
real-valued functions which are defined on $\mr$ and have finite
limits from the left at each point of $\mr$. Let
$(\Omega,\mathcal{F},\mmp)$ be a probability space and let
$(X,\xi)$ be a random element in $(\Omega,\mathcal{F},\mmp)$ with
values in $D(\mr)\times [0,\infty)$. More precisely,
$X:=(X(t))_{t\in\mr}$ is a
random process with paths in $D(\mr)$ which satisfies 
$X(t)=0$ for all $t < 0$, and $\xi$ is a positive random variable.
$X$ and $\xi$ are allowed to be arbitrarily dependent. Further, on
$(\Omega,\mathcal{F},\mmp)$ define the sequence $(X_1,\xi_1),
(X_2, \xi_2),\ldots$ of i.i.d.~copies of the pair $(X,\xi)$. Let
$(S_n)_{n\in\mn_0}$ (we use the notation $\mn_0$ for the set of
nonnegative integers $\{0,1,2,\ldots\}$) be the zero-delayed
random walk with increments $\xi_k$, i.e.,
\begin{equation*}
S_0 := 0, \qquad    S_n:=\xi_1+\ldots+\xi_n, \quad  n\in\mn.
\end{equation*}
Denote by $(\nu(t))_{t\in\mr}$ the associated first-passage time
process given by $\nu(t):=\inf\{k\in\mn_0: S_k>t\}$ for $t\in\mr$.
The process $Y := (Y(t))_{t\in\mr}$ defined by
\begin{equation}\label{eq:Y(t)}
Y(t) := \sum_{k\geq 0}X_{k+1}(t-S_k) = \sum_{k=0}^{\nu(t)-1}X_{k+1}(t-S_k), \quad   t\in\mr
\end{equation}
is called {\it random process with immigration at the epochs of a
renewal process} or just {\it random process with immigration}.
These processes were introduced in
\cite{Iksanov+Marynych+Meiners:2015-1,Iksanov+Marynych+Meiners:2015-2}.
Since 
the paper at hand is intended to be a part of this series, we
refrain from giving a detailed discussion here and refer the
reader to the introduction in
\cite{Iksanov+Marynych+Meiners:2015-1} for the motivation,
bibliographic comments as well as the explanation of the term
``processes with immigration''.

In this note we are interested in weak convergence of random
processes with immigration in the situation when $\me [|X(t)|]$ is
finite and, in some sense to be specified later, integrable on
$[0,\infty)$ and $\me \xi<\infty$. Under these assumptions and the
additional assumption that $X$ and $\xi$ are independent, a
functional limit theorem has been obtained recently in
\cite{Iksanov+Marynych+Meiners:2015-2}. In the present paper we
prove a counterpart of that result discarding the independence
assumption.


Before we formulate our main result, some preliminary work has to be done. In Sections \ref{d_r_conv_subsection} and \ref{subsec_mpp}  
we recall some necessary information regarding the Skorokhod space
$D(\mr)$ and the space of marked point processes. In Section
\ref{subsec_smrp} we present the construction of the stationary
marked renewal process.

\subsection{Convergence in $D(\mr)$}\label{d_r_conv_subsection}\footnote{This material 
was borrowed from \cite{Iksanov+Marynych+Meiners:2015-2} and is
given here for the ease of reference.} Consider the subset $D_0$
of the Skorokhod space $D(\mr)$ composed of those functions $f\in
D(\mr)$ which have finite limits
$f(-\infty):=\lim_{t\to-\infty}f(t)$ and
$f(\infty):=\lim_{t\to+\infty}f(t)$. For $a,b\in\mr$, $a<b$ let
$d_0^{a,b}$ be the Skorokhod metric on $D[a,b]$, i.e.,
$$  d_0^{a,b}(x,y)
= \inf_{\lambda \in \Lambda_{a,b}} \bigg( \sup_{t \in [a,b]} |x(\lambda(t)) - y(t)| \vee \sup_{s \not = t} \Big| \log \Big(\frac{\lambda(t)-\lambda(s)}{t-s}\Big)\Big|\bigg)$$
where $\Lambda_{a,b} = \{\lambda: \lambda \text{ is a strictly increasing and continuous function on } [a,b] \text{ with } \lambda(a)=a, \lambda(b)=b\}.$
Following \cite[Section 3]{Lindvall:1973}, for
$f,g\in D_0$, put
\begin{equation*}
d_0(f,g):=d_0^{0,1}(\overline{\phi}(f),\overline{\phi}(g)),
\end{equation*}
where $$\phi(t):=\log (t/(1-t)),\;\;t\in
(0,1),\;\;\phi(0)=-\infty,\;\;\phi(1):=+\infty$$ and
$$
\overline{\phi}:D_0 \to D[0,1],\;\;\overline{\phi}(x)(\cdot):=x(\phi(\cdot)),\;\;x\in D_0.
$$
Then $(D_0,d_0)$ is a complete separable metric space. Mimicking
the argument given in \cite[Section 4]{Lindvall:1973} and using
$d_0$ as a basis one can construct a metric $d$ (its explicit form
is of no importance here) on $D(\mr)$ such that $(D(\mr),d)$ is a complete
separable metric space. We shall need the following
characterization of the convergence on $(D(\mr),d)$, see Theorem 1(b)
in \cite{Lindvall:1973} and Theorem 12.9.3(ii) in
\cite{Whitt:2002} for the convergence on $D[0,\infty)$.
\begin{prp}\label{conv_in_d}
Suppose $f_n, f\in D(\mr)$, $n\in\mn$. The following conditions are
equivalent:
\begin{itemize}
\item[(i)] $f_n\to f$ in $(D(\mr),d)$ as $n\to\infty$;
\item[(ii)] there exist $\lambda_n \in \Lambda$, where
$$\Lambda:=\{\lambda : \lambda\text{ is a strictly increasing and continuous function on } \mr \text{ with } \lambda(\pm\infty)=\pm\infty\},$$
such that, for any finite $a$ and $b$, $a<b$,
$$ \lim_{n\to\infty} \max\Big\{\sup_{u\in [a,\,b]}
|f_n(\lambda_n(u))-f(u)|,\,\sup_{u\in [a,\,b]}|\lambda_n(u)-u|\Big\}=0;
$$
\item[(iii)] for any finite $a$ and $b$, $a<b$ which are continuity points of $f$
it holds that $f_n|_{[a,\,b]}\to f|_{[a,\,b]}$ in
$(D[a,b],d_0^{a,b})$ as $n\to\infty$, where $g|_{[a,\,b]}$ denotes
the restriction of $g\in D(\mr)$ to $[a,b]$.
\end{itemize}
\end{prp}

\subsection{Marked point processes and their convergence}\label{subsec_mpp}
In this section we recall the notion of a marked point process on
$\mr$ along with the corresponding canonical spaces. We refer to
the books \cite{Matthes+Kerstan+Mecke:1978} and \cite{Sigman:1995}
for the comprehensive exposition of the theory of marked point
processes.

Let $(K,\rho_K)$ be an arbitrary complete separable metric space
and let $(\mr\times K,\rho)$ be the product of $\mr$ and $K$
endowed with the product topology:
$$
\rho((x_1,k_1),(x_2,k_2))=|x_1-x_2|+\rho_K(k_1,k_2),\quad x_1,x_2\in\mr,\quad k_1,k_2\in K.
$$
Let $M_K$ be the set of integer-valued measures $m$ on $(\mr\times
K,\mathcal{B}(\mr\times K))$ such that $m(\mr\times K)=\infty$
and, for every bounded set $A\in \mathcal{B}(\mr)$,
$$
m(A\times K)<\infty,\quad
$$
where $\mathcal{B}(\mathbb{X})$ denotes the Borel sigma-algebra of
the metric space $\mathbb{X}$. An arbitrary $m\in M_K$ can be
represented as a countable sum of Dirac point measures on
$\mr\times K$:
\begin{equation}\label{mpp_rep}
m=\sum_{n\in\mz} \delta_{(t_n,k_n)},
\end{equation}
where the first coordinates can be arranged in the non-decreasing order:
\begin{equation}\label{arrival_epochs_nondecreasing}
\cdots\leq t_{-2}\leq t_{-1} < 0\leq t_0\leq t_1\leq t_2\cdots.
\end{equation}
The elements of $M_K$ are called {\it marked point processes}
(``mpp'' in what follows). A mpp $m$ is called simple if the
corresponding sequence $(t_n)_{n\in\mz}$ is strictly increasing in
which case representation \eqref{mpp_rep} subject to constraints
\eqref{arrival_epochs_nondecreasing} is unique. The sequence
$(t_n)_{n\in\mz}$ represents the arrival epochs and the sequence
$(k_n)$ represents the marks, so that $k_n$ is the mark brought by
the arrival at time $t_n$. The space $K$ is called {\it the mark
space}.

It is known (see, for instance, \cite[Chapter
1.15]{Matthes+Kerstan+Mecke:1978}) that $M_K$ can be endowed with
a structure of complete separable metric space, more precisely
there exists a metric $\rho_{M_K}$ such that:
\begin{itemize}
\item $(M_K,\rho_{M_K})$ is a complete separable metric space;
\item $\rho_{M_K}(m_n,m)\to 0$ as $n\to\infty$ iff
$$
\int f(x)m_n({\rm dx})\to \int f(x)m({\rm dx}),\quad n\to\infty
$$ for every continuous function $f:\mr\times K\to \mr^{+}$ with bounded support.
\end{itemize}

The following proposition gives another characterization of the convergence on space $(M_k,\rho_{M_K})$ (see Theorem D.1 and Corollary D.2 in \cite[Appendix D]{Sigman:1995}).
\begin{prp}\label{conv_of_mpp}
A sequence $m_n:=\sum_{j\in\mz}\delta_{(t^{(n)}_j,k_j^{(n)})}$, $n\in\mn$, of simple mpp's with the increasing enumeration of arrival epochs
$$\cdots<t^{(n)}_{-2}<t^{(n)}_{-1}<0\leq t^{(n)}_0<t^{(n)}_1<t^{(n)}_2<\cdots,\quad n\in\mn$$
converges in $(M_K,\rho_{M_K})$ to a simple mpp
$m:=\sum_{j\in\mz}\delta_{(t_j,k_j)}$ satisfying
\eqref{arrival_epochs_nondecreasing} iff
$$((t^{(n)}_{-q},k_{-q}^{(n)}),\ldots, (t^{(n)}_p,k_p^{(n)}))\to ((t_{-q},k_{-q}),\ldots, (t_{p},k_{p})),\quad
n\to\infty$$ for every $p,q\in\mn$.
\end{prp}

\subsection{Stationary marked renewal point processes and stationary random processes with immigration}\label{subsec_smrp}
Suppose that $\mu := \me \xi<\infty$ and that the distribution of $\xi$ is nonlattice, i.e., it is not
concentrated on any lattice $d\mz$, $d>0$. On the probability space $(\Omega,\mathcal{F},\mmp)$, where the sequence $(\xi_k,X_k)_{k\in\mn}$ lives,
define the following objects:
\begin{itemize}
\item an independent copy $(X_{-k},\xi_{-k})_{k\in\mn}$ of $(X_k,\xi_k)_{k\in\mn}$;
\item a pair $(X_0,\xi_0)$ which is independent of $(X_k,\xi_k)_{k\in\mz\setminus\{0\}}$ and has joint distribution
\begin{equation}\label{zero_interval_joint_distribution}
\mmp\{\xi_0\leq x,X_0\in \cdot\}=\frac{1}{\mu}\int_{[0,\,x]}
y\mmp\{\xi\in{\rm d}y,X\in \cdot\},\quad x\geq 0;
\end{equation}
\item a random variable $U$, independent of $(X_k,\xi_k)_{k\in\mz}$, with the uniform distribution on $[0,1]$.
\end{itemize}
Set
$$
S_{-k}:=-(\xi_{-1}+\ldots+\xi_{-k}),\quad k\in\mn
$$
and
$$
S_0^{\ast}:=U\xi_0,\quad S^{\ast}_{-1}:=-(1-U)\xi_0,\quad S^{\ast}_k:=S^{\ast}_0+S_k,\quad S^{\ast}_{-k-1}:=S_{-1}^{\ast}+S_{-k},\quad k\in\mn.
$$
The (unmarked) point process
$$
\mathcal{S}:=\sum_{n\in\mz}\delta_{S^{\ast}_n}
$$
is called {\it stationary renewal point process}. Its properties
are well understood. In particular, it is shift invariant, i.e.,
$\sum_{n\in\mz}\delta_{S^{\ast}_n+t}$ has the same distribution as
$\sum_{n\in\mz}\delta_{S^{\ast}_n}$ for all $t\in\mr$. Its
intensity measure is
\begin{equation}\label{eq:intensity=Lebesgue}
\me \mathcal{S}(A)=\mu^{-1}|A|,\quad A\in\mathcal{B}(\mr),
\end{equation}
where $|A|$ is the Lebesgue measure of $A$. Now consider the
process $X_k$ as a mark brought by the point that arrived at time
$S_k^{\ast}$, so that the mark space $K$ is $D(\mr)$. It is
natural to call the mpp
\begin{equation}\label{marked_stationary_definition}
\mathcal{S}^{\mathcal{M}}:=\sum_{n\in\mz}\delta_{(S^{\ast}_n,\,X_{n})}
\end{equation}
{\it two-sided stationary marked renewal process}. Note that it is
simple with probability one in view of the assumption
$\mmp\{\xi>0\}=1$. The process $\mathcal{S}^{\mathcal{M}}$ is
stationary (see Example 1 on pp.~27-29 in \cite{Sigman:1995} for
the one-sided case) in the following sense:
\begin{equation*}
\mathcal{S}^{\mathcal{M}}(\{A+t\}\times B)\od \mathcal{S}^{\mathcal{M}}(A\times B),
\end{equation*}
for arbitrary $A\in \mathcal{B}(\mr)$, $B\in\mathcal{B}(D(\mr))$
and $t\in\mr$, where $\{A+t\}:=\{a+t: a\in A\}$ and $\od$ denotes
equality of distributions. Equivalently,
\begin{equation}\label{mpp_is_stationary}
\sum_{k\in\mz}\delta_{(S^{\ast}_k-t,X_{k})}\od\sum_{k\in\mz}\delta_{(S^{\ast}_k,X_{k})},\quad t\in\mr.
\end{equation}
This fact is a simple consequence of Lemma \ref{rmpp_conv} below
(see Remark \ref{rem_stationarity}). It can also be deduced from
the Palm theory of stationary point processes, see e.g. Chapter
4.8 in \cite{Thorisson:2000}. 

From the construction above it is clear that the distribution of the stationary renewal point process is symmetric around the origin:
\begin{equation}\label{srp_is_symm}
\sum_{n\in\mz}\delta_{S^{\ast}_n}\od \sum_{n\in\mz}\delta_{-S^{\ast}_{-n-1}}.
\end{equation}
For every $k\in\mz$ the mark $X_k$ depends only on the interarrival time $\xi_k=S^{\ast}_k-S^{\ast}_{k-1}$. This observation together with 
\eqref{srp_is_symm} allows us to conclude that
\begin{equation}\label{smrp_is_symm}
\sum_{n\in\mz}\delta_{(S^{\ast}_n,X_{n})}\od \sum_{n\in\mz}\delta_{(S^{\ast}_{n},X_{n+1})}.
\end{equation}

Fix any $u\in\mr$. Since $\lim_{k \to -\infty} S_k^\ast=-\infty$ a.s. (almost surely), the sum
\begin{equation}
\sum_{k\leq -1}X_{k}(u+S_k^\ast)\1_{\{S_k^\ast\geq -u\}}
\end{equation}
is a.s. finite because the number of non-zero summands is
a.s. finite. Define
\begin{equation*}
Y^\ast(u):=\sum_{k\in\mz}X_{k}(u+S_k^\ast) = \sum_{k\in\mz}X_{k}(u+S_k^\ast)\1_{\{S_k^\ast \geq -u\}}
\end{equation*}
with the random variable $Y^\ast(u)$ being a.s. finite
provided that the series $\sum_{k\geq
0}X_{k}(u+S_k^\ast)\1_{\{S_k\geq -u\}}$ converges in probability.
The process $Y^\ast:=(Y^{\ast}(u))_{u\in\mr}$ is called {\it
stationary random process with immigration.} In view of
\eqref{mpp_is_stationary} the process $Y^{\ast}$ is strictly stationary in the usual sense.

Note that from \eqref{smrp_is_symm} we obtain
$$
Y^{\ast}(\cdot)\od \sum_{k\in\mz}X_{k+1}(\cdot+S_k^\ast)=\sum_{k\in\mz}X_{k+1}(\cdot+S_k^\ast)\1_{\{S_k^\ast \geq -u\}}.
$$
In \cite{Iksanov+Marynych+Meiners:2015-2} the stationary process with immigration has been defined by the
series on the right-hand side of the last equality. In some situations this representation is more convenient. For example, it shows that $Y^{\ast}(u)\od Y^{\ast}(0)\od \sum_{k\geq 0}X_{k+1}(S^{\ast}_k)$. In particular, \eqref{eq:intensity=Lebesgue} readily implies 
$$
\me Y^{\ast}(u)=\me Y^{\ast}(0)=\int_0^{\infty}\me X(t){\rm d}t,
$$
since $X_{k+1}$ is independent of $S^{\ast}_k$ for $k\geq 0$\footnote{However, for $k<0$ the mark $X_{k+1}$ depends on $S^{\ast}_k$.}.

\subsection{Main result}
In the following we write `$Z_t\Rightarrow Z$ as $t\to\infty$ on
$(S,d^\ast)$'\, to denote weak convergence of processes on a
complete separable metric space $(S,d^\ast)$ and
`$\overset{d}{\to}$' to denote convergence in distribution of
random variables or random vectors.

Let $\mathfrak{A}$ be the support of the discrete component of
$\xi$, i.e., $\mathfrak{A}:=\{t>0:\mmp\{\xi=t\}>0\}$, and set
$$
<\mathfrak{A}>:=\Big\{\sum_{i}n_ia_i:a_i\in\mathfrak{A},n_i\in\mn_0\Big\}.
$$
Observe that $a\;\in\;<\mathfrak{A}>$ iff $(S_j)_{j\in\mn_0}$ hits
$a$ with positive probability. Further, let
\begin{equation}\label{d12_def}
\mathcal{D} := \{t \in\mr: \mmp\{X(t)\neq X(t-)\}>0\},\, \mathcal{D}_{\xi} := \{t \in\mr: \mmp\{X(\xi+t)\neq X(\xi+t-)\}>0\}
\end{equation}
and
\begin{equation}\label{Delta_x_def}
\Delta_X:=\mathcal{D}_{\xi}\varominus\mathcal{D}:=\{a-b:a\in\mathcal{D}_{\xi},b\in\mathcal{D}\}.
\end{equation}

\begin{thm}\label{main1}
Suppose that $\mu:=\me\xi<\infty$ and the distribution of $\xi$ is nonlattice.
\begin{itemize}
    \item[(a)]
        If the function $G(t):=\me [|X(t)|\wedge 1]$ is directly Riemann integrable\footnote{See Chapter 3.10.1 in \cite{Resnick:1992} for the definition of direct Riemann     integrability.} (dRi) on $[0,\infty)$, then, for each
        $u\in\mr$, the series $\sum_{k\geq 0}X_{k}(u+S_k^\ast)\1_{\{S_k^\ast\geq -u\}}$
        is absolutely convergent with probability one, and, for any $n\in\mn$ and any finite $u_1<u_2<\ldots<u_n$ such that $Y^{\ast}$ is almost surely continuous at $u_i$,
        \begin{equation}\label{main_relation}
        \big(Y(t+u_1),\ldots, Y(t+u_n)\big) \ \dod \
        \big(Y^\ast(u_1),\ldots, Y^\ast(u_n)\big), \ \ t\to\infty.
        \end{equation}
    \item[(b)]
        If, for some $\varepsilon>0$, the function
        $H_{\varepsilon}(t) := \me [\sup_{u\in[t,\,t+\varepsilon]} |X(u)|\wedge 1]$ is dRi on $[0,\infty)$, and
        \begin{equation}\label{no_common_disc_ass}
        <\mathfrak{A}>\cap\; \Delta_X=\varnothing,
                \end{equation}
        then
        \begin{equation}\label{main_relation_func}
        Y(t+u) \ \Rightarrow \ Y^\ast(u), \ \ t\to\infty        \quad   \text{on } (D(\mr),d).
        \end{equation}
\end{itemize}
\end{thm}
\begin{rem}\label{remark1}
Condition \eqref{no_common_disc_ass} needs to be checked only if the distribution of $\xi$ has a discrete component. Otherwise, it holds automatically.
If $\xi$ and $X$ are independent, condition \eqref{no_common_disc_ass} can replaced by the following simpler condition
\begin{equation}\label{no_common_disc_ass2}
<\mathfrak{A}>\cap\; \Delta^{\prime}_X=\{0\},
\end{equation}
where $\Delta^{\prime}_X:=\mathcal{D}\varominus\mathcal{D}$. It
can be verified that \eqref{no_common_disc_ass2} is equivalent to
formula (2.3) in \cite{Iksanov+Marynych+Meiners:2015-2}.
\end{rem}

\section{Examples and applications}
In this section we give a few examples from several areas of
applied probability in which the random processes with immigration
appear. Further examples can be found in Section 4 of
\cite{Iksanov+Marynych+Meiners:2015-2}.

\subsection{$GI/G/\infty$ queues.} Let $(\xi,\eta)$ be a random vector with positive components and let $((\xi_k,\eta_k))_{k\in\mn}$ be a sequence of independent copies of $(\xi,\eta)$. Assuming that $\xi$ and $\eta$ are independent and interpreting $\xi_k$ as the interarrival time between $k$-th and $(k+1)$-st customer and $\eta_k$ as the service time of the $k$-th customer in the queuing system with infinitely many servers, we obtain the classical $GI/G/\infty$ queue.
However, in real world situation the assumption that $\xi_k$ and
$\eta_k$ are independent is not always adequate and one may
consider, for example, the model with positively correlated
$\xi_k$ and $\eta_k$, i.e., the larger service time of the last
arrived customer, the more time it takes for the next customer to
arrive. One of several quantities of interest is the number $Y(t)$
of busy servers at the system at time $t$
$$
Y(t)=\sum_{k\geq 0}\1_{\{S_k\leq t <S_k+\eta_{k+1}\}}
$$
which is a particular instance of the random process with
immigration with $X(t):=\1_{\{\eta>t\geq 0\}}$. Assuming that the
distribution of $\xi$ is nonlattice, has finite mean and condition
\eqref{no_common_disc_ass} is fulfilled we can apply our Theorem
\ref{main1} to deduce that \eqref{main_relation_func} holds iff
$\me\eta<\infty$. If the latter holds the queuing system possesses
a stationary regime. In the case of independent $\xi$ and $\eta$ a
one-dimensional version of this result has initially been proved
in \cite{Kaplan:1975}.

\subsection{Divergent perpetuities.} Let $(\xi,\eta)$ be a random vector with $\mmp\{\xi>0\}=1$ and put $X(t):=\eta e^{-at}$, $t\geq 0$,
where $a>0$ is a fixed constant. Then, for every fixed $t\geq 0$, $Y(t)$ is the normalized truncated perpetuity given by
$$
Y(t)=e^{-at}\sum_{k=0}^{\nu(t)-1}\eta_{k+1}e^{aS_k}=e^{-at}\sum_{k=0}^{\nu(t)-1}\eta_{k+1}\prod_{i=1}^{k}e^{a\xi_i},\quad
t\geq 0.
$$
Assuming that $\me\xi<\infty$ and that the distribution of $\xi$
is non-lattice we infer that \eqref{main_relation_func} holds if
$\me (\log^{+}|\eta|)<\infty$.
Indeed, since 
$\mathcal{D}=\{0\}$ and $\mathcal{D}_{\xi}=-\mathfrak{A}$ we
conclude that $\Delta_X=-\mathfrak{A}$ which implies that
\eqref{no_common_disc_ass} holds. Further, for every
$\varepsilon>0$, the function
$$
H_{\varepsilon}(t)=\me[\sup_{u\in [t,t+\varepsilon]}|X(u)|\wedge 1]=\me[|\eta| e^{-at}\wedge 1],\quad t\geq 0
$$
is non-increasing, hence dRi iff it is Lebesgue integrable iff
$\me (\log^{+}|\eta|)<\infty$. Whenever the latter holds 
the one-dimensional distribution of the limiting process can be
characterized via the distribution of a suitable convergent
perpetuity. 
By stationarity of $Y^{\ast}$ we have, for $u>0$,
$$
Y^{\ast}(u)\od Y^{\ast}(0)=\sum_{k\geq
0}\eta_ke^{-aS^{\ast}_k}=\eta_0e^{-aS^{\ast}_0}+e^{-aS_{0}^{\ast}}\sum_{k\geq
1}\eta_ke^{-aS_{k}}=:
\eta_0e^{-aS^{\ast}_0}+e^{-aS_{0}^{\ast}}A_{\infty}.
$$
Here $(\eta_0,S_0^{\ast})$ is independent of $A_{\infty}$, a
perpetuity which satisfies the distributional equality
$$
A_{\infty}\od \eta e^{-a\xi}+e^{-a\xi}A^{\prime}_{\infty},
$$
where $A^{\prime}_{\infty}\od A_{\infty}$ and
$A^{\prime}_{\infty}$ is independent of $(\xi,\eta)$.


\subsection{Continuous time random walks.}  A continuous time random walk ({\it ctrw}, in short) $(C(t))_{t\in\mr}$ is defined by
$$
C(t)=\sum_{k=0}^{\nu(t)-1}J_{k+1},
$$
where $(\nu(t))_{t\in\mr}$ is as before and $(J_k)_{k\in\mn}$ is a
sequence of random displacements which is usually assumed to be
comprised of i.i.d. random variables. $C(t)$ is then interpreted
as the position at time $t$ of a particle performing jumps of
sizes $J_{k+1}$ at the epochs $S_k$, $k=0,\ldots,\nu(t)-1$. If
$J_k$ and $\xi_k$  are independent for every $k\in\mn$ the {\it
ctrw} is called uncoupled, otherwise it is coupled. For the
up-to-date exposition of the limit theory for {\it ctrw} we refer
to \cite{Meerschaert+Straka:2012} and references therein.

If $X$ in the definition of the random process with immigration 
takes the form $X(t)\equiv \eta$, $t\geq 0$, for some random
variable $\eta$, then $Y$ is the classical {\it ctrw}. In general,
the random process with immigration can be thought of as a
generalized ctrw with time-inhomogeneous jumps in which the
displacement size depends on the time. 
The situation considered in our Theorem \ref{main1} covers a very
special class of such ``generalized coupled continuous time random
walks'' with finite mean of the waiting times and rapidly
decreasing displacement process $X$.

\section{Proof of Theorem \ref{main1}}\label{prmain1}
Our proof relies on the sequence of auxiliary lemmas given next
along with the continuous mapping arguments. In what follows we
write $M$ as a shorthand for $M_{D(\mr)}$ and $\rho_M$ for the
corresponding metric so that $(M_{D(\mr)},\rho_m)$ is 
a complete separable metric space (see Section \ref{subsec_mpp}
above).

\begin{lem}\label{rmpp_conv}
Assume that $\me\xi<\infty$ and that the distribution of $\xi$ is nonlattice. Then, as $t\to\infty$,
\begin{equation}\label{conv_of_rmpp}
\sum_{k\in\mz}\delta_{(t-S_k,X_{k+1})} \ \Rightarrow \mathcal{S}^{\mathcal{M}}=\sum_{k\in\mz}\delta_{(S^\ast_k,X_k)}
\end{equation}
on $(M,\rho_M)$.
\end{lem}

In view of Lemma 5.1 in \cite{Iksanov+Marynych+Meiners:2015-2},
the result stated in Lemma \ref{rmpp_conv} is quite expected.
However, the rigorous proof is technically involved. Hence it is
relegated to the Appendix.

\begin{rem}\label{rem_stationarity}
Fix $u\in\mr$. The shift-mapping $\theta_u:M\to M$ defined by $\theta_u(m(A\times B))=m(\{A+u\}\times B)$, $A\in\mathcal{B}(\mr)$, $B\in\mathcal{M}(D(\mr))$ is obviously continuous with respect to $\rho_M$ for every $u\in\mr$. Therefore, as $t\to\infty$,
\begin{eqnarray*}
\sum_{k\in\mz}\delta_{(S^\ast_k,X_k)}\Leftarrow \sum_{k\in\mz}\delta_{(t-u-S_k,X_{k+1})}&=&\theta_u\Big(\sum_{k\in\mz}\delta_{(t-S_k,X_{k+1})}\Big)\\
&\Rightarrow&
\theta_u\Big(\sum_{k\in\mz}\delta_{(S^\ast_k,X_k)}\Big)=\sum_{k\in\mz}\delta_{(S^\ast_k-u,X_k)}
\end{eqnarray*}
on $(M,\rho_M)$ which shows that $\mathcal{S}^{\mathcal{M}}$ is
stationary.
\end{rem}

\begin{lem}   \label{continuity_lemma1}
The mapping $T:\mr\times D(\mr)\to D(\mr)$ defined by
$$
T(t_0,f(\cdot))=f(t_0+\cdot)
$$
is measurable with respect to $\mathcal{B}(\mr\times D(\mr))$ and $\mathcal{B}(\mr)$ and continuous on $\mr\times D(\mr)$ endowed with the product topology.
\end{lem}
\begin{proof}
To show the measurability note that by Lemma 2.7 in \cite{Whitt:1980} it is enough to show that $(t_0,f(\cdot))\mapsto f(t_0+t)$ is measurable for
every $t\in\mr$, but this is obvious since this mapping is a composition of two measurable mappings $(t_0,t)\mapsto t_0+t$ and $f\mapsto f(t)$.
The continuity has been proved in \cite{Iksanov+Marynych+Meiners:2015-2}, see Lemma 5.2 therein.
\end{proof}

For fixed $c>0$, $l\in\mn$ and $(u_1,\ldots,u_l)\in\mr^l$, define the mapping
$\phi^{(l)}_c:M \to \mr^l$ by
\begin{equation*}
\phi^{(l)}_c\Big(\sum_{n\in\mz}\delta_{(t_n,f_n(\cdot))}\Big):=\bigg(\sum_{n\in\mz}f_n(t_n+u_j)\1_{\{|t_n|\leq
c\}}\bigg)_{j=1,\ldots,l}
\end{equation*}
and the mapping $\phi_c:M \to D(\mr)$ by
\begin{equation*}
\phi_c\Big(\sum_{n\in\mz}\delta_{(t_n,f_n(\cdot))}\Big) :=
\sum_{n\in\mz}f_n(t_n+\cdot)\1_{\{|t_n| \leq c\}}.
\end{equation*}
For $f\in D(\mr)$, denote by ${\rm Disc}(f)$ the set of
discontinuity points of $f$ on $\mr$. Clearly, both $\phi^{(l)}_c$ and $\phi_c$ are measurable as finite sums of measurable mappings.
The continuity is provided by the next lemma.

\begin{lem}   \label{continuity_lemma2}
The mapping $\phi^{(l)}_c$ is continuous at all points
$m=\sum_{n\in\mz}\delta_{(t_n,f_n(\cdot))}$ such that $m$ is simple, $m(\{-c,c\}\times D(\mr))=0$ and for
which $u_1,\ldots,u_l$ are continuity points of $f_k(t_k+\cdot)$
for all $k\in\mz$. $\phi_c$ is continuous at all points
$m=\sum_{n\in\mz}\delta_{(t_n,f_n(\cdot))}$ such that $m$ is simple, $m(\{-c,c\}\times D(\mr))=0$ and ${\rm
Disc}(f_k(t_k+\cdot))\cap {\rm Disc}(f_j(t_j+\cdot))=\varnothing$
for $k\neq j$.
\end{lem}
\begin{proof}
Let $c>0$ and suppose that
\begin{equation}\label{lemma_conv4}
\sum_{j\in\mz}\delta_{(t_j^{(n)},f_j^{(n)})}=:m_n\to m=:\sum_{j\in\mz}\delta_{(t_j,f_j)},\quad n\to\infty
\end{equation}
on $(M,\rho_M)$ where $m(\{-c,c\}\times D(\mr))=0$. Let $p,q\in\mn_0$ be such that
\begin{equation}\label{limit_is_also_in_pq}
t_{-p-1}<-c<t_{-p},\quad\quad t_q < c < t_{q+1}.
\end{equation}
From \eqref{lemma_conv4}, using Proposition \ref{conv_of_mpp}, we conclude that for $-p\leq k\leq q$
\begin{equation}\label{lemma_conv41}
(t_k^{(n)},f_k^{(n)})\to (t_k,f_k),\quad n\to\infty
\end{equation}
on $\mr\times D(\mr)$ and also for large enough $n$
\begin{equation}\label{pre_limit_is_also_in_pq}
t^{(n)}_{-p-1}<-c<t^{(n)}_{-p},\quad\quad t^{(n)}_q < c < t^{(n)}_{q+1}.
\end{equation}
By Lemma \ref{continuity_lemma1} convergence \eqref{lemma_conv41} yields
\begin{equation}\label{1221}
f_k^{(n)}(t_k^{(n)}+\cdot)\to f_k(t_k+\cdot),\quad n\to\infty
\end{equation}
on $D(\mr)$.

Now assume that $u_1,\ldots,u_l$ are continuity points of $f_k(t_k+\cdot)$ for all $k\in\mz$, then \eqref{1221} implies that
$$  \big(f^{(n)}_k(t^{(n)}_k+u_1),\ldots,f^{(n)}_k(t^{(n)}_k+u_l)\big)
~\to~   \big(f_k(t_k+u_1),\ldots, f_k(t_k+u_l)\big),\;\;n\to
\infty $$ for $-p\leq k\leq q$. Summation of these relations over
$k=-p,\ldots,q$ proves the continuity of $\phi_c^{(l)}$ in view of
\eqref{limit_is_also_in_pq} and \eqref{pre_limit_is_also_in_pq}.

Theorem 4.1 in \cite{Whitt:1980} tells us that addition on
$D(\mr)\times D(\mr)$ is continuous at those $(x,y)$ for which
${\rm Disc}(x)\cap {\rm Disc}(y)=\varnothing$. Since this
immediately extends to any finite number of summands we conclude
that relations \eqref{1221} entail $$ \phi_c(m_n)=\sum_{k=-p}^{q}
f_k^{(n)}(t_k^{(n)}+\cdot) \ \to \ \sum_{k=-p}^{q}
f_k(t_k+\cdot)=\phi_c(m),\;\;n\to\infty$$ on $D(\mr)$ provided
that ${\rm Disc}(f_k(t_k+\cdot))\cap {\rm
Disc}(f_j(t_j+\cdot))=\varnothing$ for $k\neq j$.
\end{proof}
\begin{rem}
The assumption that $m$ is simple could have been omitted at the expense of a more involved proof. 
The present version of Lemma \ref{continuity_lemma2} serves our
needs.
\end{rem}

The next lemma relates the integrability of $X$ to pathwise
properties of $Y^{\ast}$, the stationary process with immigration.

\begin{lem}\label{limit_process_properties}
Assume that $\me \xi<\infty$ and that the law of $\xi$ is
nonlattice. 
\begin{itemize}
\item[(i)] If $G(t)=\me[|X(t)|\wedge 1]$ is Lebesgue integrable on $[0,\infty)$, then
$|Y^{\ast}(u)|<\infty$ for every $u\in\mr$ almost surely.
\item[(ii)] If, for some (hence all) $\varepsilon>0$, the function $H_{\varepsilon}(t)=\me[\sup_{u\in[t,t+\varepsilon]}|X(u)|\wedge 1]$ is dRi
on $[0,\infty)$, then $Y^{\ast}(u)$ takes values in $D(\mr)$
almost surely.
\end{itemize}
\end{lem}
\begin{proof}[Proof of (i)]
Put 
$\widehat{G}(t):=\me[|X(\xi+t)|\wedge 1]$ and note that
\begin{eqnarray*}
\int_0^{\infty}\me[|X(\xi+t)|\wedge 1]{\rm d}t=\me\int_{0}^{\infty}[|X(t)|\wedge 1]{\rm d}t-\me\int_0^{\xi}[|X(t)|\wedge 1]{\rm d}t.
\end{eqnarray*}
Since $\me\int_0^{\xi}[|X(t)|\wedge 1]{\rm d}t<\me\xi<\infty$ we see that the Lebesgue integrability of $G$ is equivalent to the Lebesgue integrability of $\widehat{G}$. Fix $u\in\mr$ and set $ \mathcal{Z}_k:=X_{k}(u+S_k^\ast)\1_{\{S_k^\ast\geq -u\}}$, $k\in\mn$. We have
\begin{eqnarray*}
\sum_{k \in \mn} \me [|\mathcal{Z}_k|\wedge 1] & = &
\sum_{k \in \mn} \me [(|X_{k}(\xi_k+u+S_{k-1}^\ast)|\wedge 1)\1_{\{S_{k}^\ast\geq-u\}}]    \notag \\
&=&\sum_{k \in \mn} \me [(|X_{k}(\xi_k+u+S_{k-1}^\ast)|\wedge 1)\1_{\{S_{k-1}^\ast\geq-u\}}]    \notag \\
&+&\sum_{k \in \mn} \me [(|X_{k}(\xi_k+u+S_{k-1}^\ast)|\wedge 1)\1_{\{S_{k}^\ast\geq-u>S_{k-1}^{\ast}\}}]    \notag \\
&\leq&\sum_{k \in \mz} \me [\widehat{G}(u+S_{k-1}^\ast)]\1_{\{S_{k-1}^\ast\geq-u\}}]+\sum_{k \in \mn} \mmp\{S_{k}^\ast\geq-u>S_{k-1}^{\ast}\}    \notag \\  \notag \\
&=&\frac{1}{\mu}\int_0^{\infty}\widehat{G}(s){\rm d}s+\mmp\{-u>S_0^{\ast}\}
\end{eqnarray*}
having utilized \eqref{eq:intensity=Lebesgue} for the last equality and independence of $X_k$ and $S^{\ast}_{k-1}$ for $k\in\mn$. Therefore the following two series
$$
\sum_{k\in\mn}\mmp\{|\mathcal{Z}_k|\geq 1\}\quad\text{and}\quad
\sum_{k\in\mn}\me (|\mathcal{Z}_k|\1_{\{|\mathcal{Z}_k|\leq 1\}})
$$
converge. The convergence of the first series together with the
Borel-Cantelli lemma implies that $|\mathcal{Z}_k|\geq 1$ for only
finitely many $k$ a.s., while the convergence of the
second implies that
$\sum_{k\in\mn}|\mathcal{Z}_k|\1_{\{|\mathcal{Z}_k|\leq
1\}}<\infty$ a.s. Hence $\sum_{k\in\mn}| \mathcal{Z}_k|
< \infty$ a.s. and $|Y^\ast(u)|<\infty$ a.s.
since $\sum_{k\leq 0}\mathcal{Z}_k$ contains only finitely many
summands for every fixed $u\in\mr$.

{\it Proof of (ii).} Again we start with the proof that dRi of $H_{\varepsilon}$ is equivalent to dRi of $\widehat{H}_{\varepsilon}(t):=\me[\sup_{u\in[t,t+\varepsilon]}|X(\xi+u)|\wedge 1]$. Note that $X(\xi+\cdot)$ with probability one takes values in $D(\mr)$, hence, using exactly the same arguments as in the second paragraph of Section 3 in \cite{Iksanov+Marynych+Meiners:2015-2}, we conclude that dRi of $\widehat{H}_{\varepsilon}$ is equivalent to
\begin{equation}\label{series_convergence}
\sum_{k\geq 0}\me \Big[\sup_{t\in[k,k+1]}(|X(\xi+t)|\wedge 1)\Big]<\infty.
\end{equation}
We have\footnote{We use the notation $\lfloor x \rfloor:=\sup\{k\in\mz : k\leq x\}$, $x\in\mr$.}
\begin{eqnarray*}
&&\hspace{-2cm}\sum_{k\geq 0}\me \Big[\sup_{t\in[k,k+1]}(|X(\xi+t)|\wedge 1)\Big]=\sum_{k\geq 0}\me \Big[\sup_{t\in[\xi+k,\xi+k+1]}(|X(t)|\wedge 1)\Big]\\
&\leq&\sum_{k\geq 0}\me \Big[\sup_{t\in[\lfloor\xi\rfloor+k,\lfloor\xi\rfloor+k+2]}(|X(t)|\wedge 1)\Big]=\me \sum_{k\geq \lfloor\xi\rfloor}\Big[\sup_{t\in[k,k+2]}(|X(t)|\wedge 1)\Big]\\
&=&\me \sum_{k\geq 0}\Big[\sup_{t\in[k,k+2]}(|X(t)|\wedge 1)\Big]-\me \sum_{k=0}^{\lfloor\xi\rfloor-1}\Big[\sup_{t\in[k,k+2]}(|X(t)|\wedge 1)\Big]\\
\end{eqnarray*}
Since $\me \sum_{k=0}^{\lfloor\xi\rfloor-1}\Big[\sup_{t\in[k,k+2]}(|X(t)|\wedge 1)\Big]\leq \me \lfloor\xi\rfloor\leq \me\xi<\infty$ we conclude that
\eqref{series_convergence} is equivalent to
$$
\me \sum_{k\geq 0}\Big[\sup_{t\in[k,k+2]}(|X(t)|\wedge
1)\Big]<\infty
$$
which, in its turn, is equivalent to dRi of $H_{\varepsilon}$ (see
formulae (3.2) and (3.3) in
\cite{Iksanov+Marynych+Meiners:2015-2}). Hence
\eqref{series_convergence} holds.

In order to show that $Y^{\ast}$ takes values in $D(\mr)$ with
probability one, we use the fact that locally uniform limits of
elements from $D(\mr)$ are again in $D(\mr)$. Therefore it is
enough to check that the series
$Y^\ast(u)=\sum_{k\in\mz}X_{k}(u+S_k^\ast)$ converges uniformly on
every compact interval a.s. which is a consequence of 
the a.s. convergence of the series
\begin{equation}\label{series_converges_uniformly}
\sum_{k\in\mz} \sup_{u \in [a,b]} |X_{k}(u+S_{k}^\ast)|
\end{equation}
for any fixed $a<b$. Arguing in the same vein as in the proof of
part (i) above, we see that it suffices to show
\begin{equation*}
\me\sum_{k\in\mz} \sup_{u \in [a,b]} (|X_{k}(\xi_k+u+S_{k-1}^\ast)| \wedge 1)<\infty,
\end{equation*}
In view of the independence of $X_k$ and $S^{\ast}_{k-1}$ for $k\in\mn$, we observe that
\begin{eqnarray}
&&\me\bigg[\sum_{k\in\mn} \sup_{u \in [a,b]} (|X_{k}(\xi_k+u+S_{k-1}^\ast)| \wedge 1)\bigg]\notag \\
&&\hspace{3cm}=\me\bigg[ \sum_{k\in\mn}  \me\bigg[\sup_{u \in [a,b]} (|X(\xi+u+S_{k-1}^\ast)| \wedge 1) \,\Big|\, S^\ast_{k-1}\bigg]\bigg]   \notag  \\
&&\hspace{3cm}\leq \me\bigg[ \sum_{k\in\mz}  \me\bigg[\sup_{u \in [a,b]} (|X(\xi+u+S_{k-1}^\ast)| \wedge 1) \,\Big|\, S^\ast_{k-1}\bigg]\bigg]\notag \\
&&\hspace{3cm}=\frac{1}{\mu} \int_{\mr} \me\bigg[\sup_{u \in [a,b]} (|X(\xi+u+t)| \wedge 1) \bigg]\, {\rm d}t  \notag  \\
&&\hspace{3cm}\leq\frac{1}{\mu} \sum_{k\in\mz}\sup_{t\in[k,\,k+1)} \me \bigg[\sup_{u\in[a+t,\,b+t]}|X(\xi+u)|\wedge 1\bigg] \notag \\
&&\hspace{3cm}\leq {\rm const}\sum_{k\in\mz} \me
\bigg[\sup_{u\in[k,\,k+1)}|X(\xi+u)|\wedge 1\bigg],
\label{eq:sup_k sup_a,b 1st}
\end{eqnarray}
where the last equality is a consequence of \eqref{eq:intensity=Lebesgue} and the last inequality follows from (5.9) in \cite{Iksanov+Marynych+Meiners:2015-2}.
The last series converges in view of \eqref{series_convergence}. It remains to note that
$$
\me\sum_{k<0} \sup_{u \in [a,b]} (|X_{k}(\xi_k+u+S_{k-1}^\ast)| \wedge 1)\leq \sum_{k<0} \mmp\{b+S_k^{\ast}\geq 0\}<\infty.
$$
Thus $\sum_{k\in\mz} X_{k}(u+S_k^\ast)$ converges uniformly on $[a,b]$ for all $a<b$ a.s. and therefore is $D(\mr)$-valued a.s.
\end{proof}

\begin{lem}\label{lemma_no_common_disc}
If \eqref{no_common_disc_ass} holds, then
$$
\mmp\{{\rm Disc}(X_i(S_i^{\ast}+\cdot))\cap {\rm Disc}(X_j(S_j^{\ast}+\cdot))\neq \varnothing\}=0,\quad i,j\in\mz,\quad i> j.
$$
\end{lem}
    \begin{proof}
Define the following random sets
$$
D^{(i)}_{\xi}:={\rm Disc}(X_i(\xi_i+\cdot)),\quad D^{(j)}:={\rm Disc}(X_j(\cdot)),\quad D^{(i,j)}:=D^{(i)}_{\xi}\varominus D^{(j)},\quad i,j\in\mz
$$
and observe that for every $t\in\mr$ the events $\{t\in D^{(i)}_{\xi}\}$, $\{t\in D^{(j)}\}$ and $\{t\in  D^{(i,j)}\}$ are measurable.
Put
\begin{eqnarray*}
&&\hspace{-2cm}\mathcal{D}^{(i)}_{\xi}:=\{t \in\mr : \mmp\{t\in D^{(i)}_{\xi}\}>0\},\quad \mathcal{D}^{(j)}:=\{t \in\mr : \mmp\{t\in D^{(j)}\}>0\},\\
&&\mathcal{D}^{(i,j)}:=\{t \in\mr : \mmp\{t\in D^{(i,j)}\}>0\},\quad i,j\in\mz.
\end{eqnarray*}
Note that by definition
\begin{equation}\label{is_null}
\mmp\{t\in D^{(i)}_{\xi}\setminus\mathcal{D}_i\}=\mmp\{t\in D^{(j)}\setminus\mathcal{D}_j\}=\mmp\{t\in D^{(i,j)}\setminus\mathcal{D}^{(i,j)}\}=0,\quad i,j\in\mz,
\end{equation}
for every fixed $t\in\mr$. We have for $i>j$
\begin{eqnarray*}
&&\hspace{-0.7cm}\mmp\{{\rm Disc}(X_i(S_i^{\ast}+\cdot))\cap {\rm Disc}(X_j(S_j^{\ast}+\cdot))\neq\varnothing\}\\
&&=\mmp\{\text{there exists }u(\omega)\text{ such that }S_i^{\ast}+u(\omega)\in {\rm Disc}(X_i(\cdot)),S_j^{\ast}+u(\omega)\in {\rm Disc}(X_j(\cdot))\}\\
&&\leq \mmp\{S_i^{\ast}-S_j^{\ast}\in {\rm Disc}(X_i(\cdot))\varominus {\rm Disc}(X_j(\cdot))\}\\
&&=\mmp\{\xi_i+\ldots+\xi_{j+1}\in {\rm Disc}(X_i(\cdot))\varominus {\rm Disc}(X_j(\cdot))\}\\
&&=\mmp\{\xi_{i-1}+\ldots+\xi_{j+1}\in {\rm Disc}(X_i(\xi_i+\cdot)\varominus {\rm Disc}(X_j(\cdot))\}\\
&&=\mmp\{\xi_{i-1}+\ldots+\xi_{j+1}\in D^{(i,j)}\}\\
&&\leq \mmp\{\xi_{i-1}+\ldots+\xi_{j+1}\in \mathcal{D}^{(i,j)}\}+\mmp\{\xi_{i-1}+\ldots+\xi_{j+1}\in D^{(i,j)}\setminus\mathcal{D}^{(i,j)}\},
\end{eqnarray*}
Since $\xi_{j+1},\ldots,\xi_{i-1}$ are independent of $D^{(i,j)}$,
the second term vanishes in view of \eqref{is_null}.

Pick an arbitrary $t_0\in \mathcal{D}^{(i,j)}$. By definition we have
\begin{eqnarray*}
0&<&\mmp\{t_0\in{\rm Disc}(X_i(\xi_i+\cdot))\varominus {\rm Disc}(X_j(\cdot))\}\\
&\leq& \1_{\{t_0\in\mathcal{D}^{(i)}_{\xi}\varominus\mathcal{D}^{(j)}\}}+\mmp\{t_0\in({\rm Disc}(X_i(\xi_i+\cdot))\setminus \mathcal{D}^{(i)}_{\xi})\varominus ({\rm Disc}(X_j(\cdot))\setminus \mathcal{D}^{(j)})\}\\
&+&\mmp\{t_0\in{\rm Disc}(X_i(\xi_i+\cdot))\varominus ({\rm Disc}(X_j(\cdot))\setminus \mathcal{D}^{(j)})\}\\
&+&\mmp\{t_0\in({\rm Disc}(X_i(\xi_i+\cdot))\setminus \mathcal{D}^{(i)}_{\xi})\varominus {\rm Disc}(X_j(\cdot))\}.
\end{eqnarray*}

The last three probabilities equal zero in view of the first two
equalities in \eqref{is_null}, the independence of
$X_i(\xi_i+\cdot)$ and $X_j(\cdot)$ and the fact that both
$D^{(j)}$ and $D_{\xi}^{(i)}$ are at most countable a.s.
Hence
$0<\1_{\{t_0\in\mathcal{D}^{(i)}_{\xi}\varominus\mathcal{D}^{(j)}\}}$
which implies
$t_0\in\mathcal{D}^{(i)}_{\xi}\varominus\mathcal{D}^{(j)}$ and
thereupon
$$
\mathcal{D}^{(i,j)}\subset
\mathcal{D}^{(i)}_{\xi}\varominus\mathcal{D}^{(j)}.
$$
Thus, we have proved that
$$
\mmp\{\xi_{i-1}+\ldots+\xi_{j+1}\in \mathcal{D}^{(i,j)}\}\leq \mmp\{\xi_{i-1}+\ldots+\xi_{j+1}\in \mathcal{D}^{(i)}_{\xi}\varominus\mathcal{D}^{(j)}\}.
$$
If $i,j\neq 0$, then $\mathcal{D}^{(i)}_{\xi}=\mathcal{D}_{\xi}$
and $\mathcal{D}^{(j)}=\mathcal{D}$ (see \eqref{d12_def} for the
definition of $\mathcal{D}_{\xi}$ and $\mathcal{D}$), and the
probability on the right-hand side of the last centered inequality
equals $\mmp\{\xi_{i-1}+\ldots+\xi_{j+1}\in \Delta_X\}$. The cases
$i=0$ or $j=0$ require a separate treatment since the joint
distribution of $(X_0,\xi_0)$ is other than that of $(X,\xi)$. We
only treat the situation when $i=0$ and show that
$\mathcal{D}^{(0)}_{\xi}\subset \mathcal{D}_{\xi}$, the other case
being similar. Assume that $t_0\notin \mathcal{D}_{\xi}$ and
consider the set $A_{t_0}:=\{f(\cdot)\in D(\mr): f(t_0)\neq
f(t_0-)\}$. Using \eqref{zero_interval_joint_distribution} we have
\begin{eqnarray*}
\mmp\{X_0(\xi_0+\cdot)\in A_{t_0}\}&=&\int_0^{\infty}\mmp\{X_0(\cdot)\in A_{t_0}\varominus\{y\},\xi_0\in {\rm d}y\}\\
&\overset{\eqref{zero_interval_joint_distribution}}{=}&\frac{1}{\mu}\int_0^{\infty}y\mmp\{\xi\in {\rm d}y,X(\cdot)\in A_{t_0}\varominus\{y\}\}\\
&=&\frac{1}{\mu}\int_0^{\infty}y\mmp\{\xi\in {\rm d}y,X(\xi+\cdot)\in A_{t_0}\}.
\end{eqnarray*}
The probability under the integral sign equals zero identically in
view of $t_0\notin \mathcal{D}_{\xi}$. Hence
$\mmp\{X_0(\xi_0+\cdot)\in A_{t_0}\}=0$ which is equivalent to
$t_0\notin \mathcal{D}_{\xi}^{(0)}$. This shows that
$\mathcal{D}^{(0)}_{\xi}\subset \mathcal{D}_{\xi}$. In the same
vein, we infer $\mathcal{D}^{(0)}\subset \mathcal{D}$ and
therefore
$$
\mmp\{\xi_{i-1}+\ldots+\xi_{j+1}\in
\mathcal{D}^{(i)}_{\xi}\varominus\mathcal{D}^{(j)}\}\leq
\mmp\{\xi_{i-1}+\ldots+\xi_{j+1}\in
\mathcal{D}_{\xi}\varominus\mathcal{D}\}
$$ provided that $i=0$ or $j=0$. Combining pieces together we obtain
$$
\mmp\{{\rm Disc}(X_i(S_i^{\ast}+\cdot))\cap {\rm Disc}(X_j(S_j^{\ast}+\cdot))\neq\varnothing\}\leq \mmp\{\xi_{i-1}+\ldots+\xi_{j+1}\in \Delta_X\}.
$$
If $j\geq 0$ or $i\leq 0$, then $\xi_{i-1}+\ldots+\xi_{j+1}\od
S_{i-j-1}$, whence $\mmp\{\xi_{i-1}+\ldots+\xi_{j+1}\in
\Delta_X\}=0$ by \eqref{no_common_disc_ass}. If $j<0<i$ the latter
equality also holds since the support of the discrete component of
$\xi_0$ is the same as that of $\xi$ by
\eqref{zero_interval_joint_distribution}. The proof of Lemma
\ref{lemma_no_common_disc} is complete.
\end{proof}

With these preparatory results at hand we are ready to prove
Theorem \ref{main1}. \vspace{0.5cm}

\noindent {\it Proof of \eqref{main_relation_func}}. We shall use
Lemma \ref{continuity_lemma2}. To this end, observe that each
$S^{\ast}_j$ has an absolutely continuous distribution. In
particular, $\mathcal{S}^{\mathcal{M}}$ is simple mpp and $\mathcal{S}^{\mathcal{M}}(\{-c,c\}\times D(\mr)) = 0$ a.s. for every $c>0$.
From Lemma \ref{continuity_lemma2} and  Lemma \ref{lemma_no_common_disc}, we see that $\phi_c$ is a.s. continuous at $\sum_{k\in\mz}\delta_{(S_k^{\ast},X_k)}$ and therefore \eqref{conv_of_rmpp} implies
\begin{eqnarray}
Y_c(t,\cdot)&:=&\sum_{k\in\mz}X_{k+1}(t-S_k+\cdot)\1_{\{|t-S_k|\leq c\}}=\phi_c\Big(\sum_{k\in\mz}\delta_{(t-S_{k},X_{k+1})}\Big)\notag \\
&\Rightarrow&
\phi_c\Big(\sum_{k\in\mz}\delta_{(S^{\ast}_{k},X_{k})}\Big)=\sum_{k\in\mz}X_k(S^{\ast}_k+\cdot)\1_{\{|S^{\ast}_k|\leq c\}}=:Y^{\ast}_c(\cdot)\label{func_conv1}.
\end{eqnarray}
on $(D(\mr),d)$.

Using Proposition \ref{conv_in_d} we conclude that in order to prove \eqref{main_relation_func} it
suffices to check that
\begin{equation}\label{main_relation_func12}
Y(t+u) \ \Rightarrow \ Y^\ast(u),\quad t\to\infty
\end{equation}
on $(D[a,b],d_0^{a,b})$ for any $a$ and $b$, $a<b$ which are
not fixed discontinuities of $Y^\ast$. To this end, first
observe that \eqref{func_conv1} implies
\begin{equation}\label{func_conv12}
Y_c(t,\cdot) \ \Rightarrow \ Y^\ast_c(\cdot),\quad t\to\infty
\end{equation}
on $(D[a,b],d_0^{a,b})$ for any $a$ and $b$, $a<b$ which are not
fixed discontinuities of $Y^\ast_c$.

From the proof of Lemma \ref{limit_process_properties} we know that the series which defines $Y^{\ast}(u)$ converges locally uniformly in $u\in\mr$. Hence for every fixed $t_0\in\mr$ we have
\begin{eqnarray}
\mmp\{t_0\in {\rm Disc}(Y^{\ast}(\cdot))\}&=&\mmp\{\text{there exists } k\in\mz\text{ such that }t_0 \in {\rm Disc}(X_k(S_k^{\ast}+\cdot))\}\notag \\
&=&\sum_{k\in\mz}\mmp\{t_0 \in {\rm Disc}(X_k(S_k^{\ast}+\cdot))\}\notag \\
&\geq& \sum_{k\in\mz}\mmp\{t_0 \in {\rm Disc}(X_k(S_k^{\ast}+\cdot)),|S^{\ast}_k|\leq c\}\notag\\
&=&\mmp\{t_0\in{\rm Disc}(Y^{\ast}_c(\cdot))\}\label{all_summands_are continuous}
\end{eqnarray}
where the second and the last equalities follow from Lemma
\ref{lemma_no_common_disc}. The process $Y^{\ast}(\cdot)$ is
stationary and is a.s. $D(\mr)$-valued by Lemma
\ref{limit_process_properties}(ii). Hence $\mmp\{t_0\in{\rm
Disc}(Y^{\ast}_c(\cdot))\}=\mmp\{t_0\in {\rm
Disc}(Y^{\ast}(\cdot))\}=0$ which implies that \eqref{func_conv12}
holds for all $a<b$.

Now \eqref{main_relation_func12}
follows from Theorem 4.2 in \cite{Billingsley:1968} if we can
prove that
\begin{equation}\label{ya_conv} Y^\ast_c \ \Rightarrow \
Y^\ast,\quad c\to\infty
\end{equation}
on $(D[a,b],d_0^{a,b})$ and that
\begin{equation}\label{bill_cond1}
\lim_{c\to\infty}\limsup_{t\to\infty}
\mmp\bigg\{d_0^{a,b}(Y_c(t,\cdot),Y(t+\cdot))>\varepsilon\bigg\}=0
\end{equation}
for all $\varepsilon>0$ and any $a,b\in\mr$, $a<b$. Since $d_0^{a,b}$ is
dominated by the uniform metric on $[a,b]$, relation \eqref{bill_cond1} follows from
\begin{eqnarray*}
&&\lim_{c\to\infty}\limsup_{t\to\infty}
\mmp\bigg\{\sup_{u\in[a,\,b]}\bigg|\sum_{k\geq 0}
X_{k+1}(u+t-S_k)\1_{\{|t-S_k| >c\}}\\
&&\hspace{5cm} +\sum_{k<0}X_{k+1}(u+t-S_k)\1_{\{|t-S_k| \leq c\}}\bigg|>\varepsilon\bigg\}=0
\end{eqnarray*}
for all $\varepsilon>0$ and any $a,b\in\mr$, $a<b$. The latter is
a consequence of 
\begin{equation}    \label{bill_cond111}
\lim_{c\to\infty}\limsup_{t\to\infty}
\mmp\bigg\{\sup_{u\in[a,\,b]}\bigg|\sum_{k\geq 0}
X_{k+1}(u+t-S_k)\1_{\{|t-S_k| >c\}}\bigg|>\varepsilon\bigg\}=0
\end{equation}
and
\begin{equation}    \label{bill_cond112}
\lim_{c\to\infty}\limsup_{t\to\infty}
\mmp\bigg\{\sup_{u\in[a,\,b]}\bigg|\sum_{k<0}X_{k+1}(u+t-S_k)\1_{\{|t-S_k| \leq c\}}\bigg|>\varepsilon\bigg\}=0.
\end{equation}
The proof of \eqref{bill_cond111}, given on p.~14 of \cite{Iksanov+Marynych+Meiners:2015-2}, works without any changes since $X_{k+1}$ and $S_k$ are independent for $k\geq 0$. The relation \eqref{bill_cond112} is trivial since
\begin{equation}\label{neg_term}
\sum_{k<0}X_{k+1}(u+t-S_k)\1_{\{|t-S_k| \leq c\}}=0
\end{equation}
for $t>c>0$.

As for relation \eqref{ya_conv} we claim an even stronger
statement
\begin{equation}\label{as_conv_func}
\sup_{u\in [a,b]}|Y^\ast_c(u)-Y^{\ast}(u)|\to 0,\quad c\to\infty
\end{equation}
a.s. for all fixed $a,b$. Indeed,
$$
\sup_{u\in [a,b]}|Y^\ast_c(u)-Y^{\ast}(u)|\leq \sum_{k\in\mz}\sup_{u\in[a,b]}|X_k(u+S^{\ast}_k)|\1_{\{|S^{\ast}_k|>c\}}
$$
Invoking the monotone convergence theorem we deduce that the
right-hand side tends to zero as $c\to\infty$ in view of
\eqref{series_converges_uniformly}.

\noindent {\it Proof of \eqref{main_relation}}. Fix $l\in\mn$ and
real numbers $\alpha_1,\ldots,\alpha_l$ and $u_1,\ldots,u_l$.
According to Lemma \ref{continuity_lemma2}, for every $c>0$, the
mapping $\phi^{(l)}_c$ is continuous at
$\big(\sum_{k\in\mz}\delta_{S^\ast_k},(X_{k+1})_{k\in\mz}\big)$
a.s. Now apply the continuous mapping theorem to
\eqref{rmpp_conv} twice (first using the map $\phi^{(l)}_c$ and
then the map $(x_1,\ldots,x_l) \mapsto \alpha_1 x_1 + \ldots +
\alpha_l x_l$) to obtain
\begin{equation*}
\sum_{i=1}^l \alpha_i Y_c(t,u_i) \ \dod \
\sum_{i=1}^l\alpha_iY_c^\ast(u_i), \ \ t\to\infty.
\end{equation*}
The proof of \eqref{main_relation} is complete if we verify
\begin{equation}\label{ya_conv_3}
\sum_{i=1}^l \alpha_i Y^\ast_c(u_i) \ \dod \ \sum_{i=1}^l \alpha_i Y^\ast(u_i),\;\;c\to\infty
\end{equation}
and
\begin{eqnarray}
&&\lim_{c\to\infty} \limsup_{t\to\infty} \mmp\bigg\{\bigg|\sum_{i=1}^l \Big(\alpha_i \sum_{k\geq 0}X_{k+1}(u_i+t-S_k)\1_{\{|t-S_k| > c\}}\notag \\
&&\hspace{5cm}+\sum_{k<0}X_{k+1}(u_i+t-S_k)\1_{\{|t-S_k| \leq c\}}\Big)\bigg|>\varepsilon\bigg\}=0\label{bill_cond_4}
\end{eqnarray}
for all $\varepsilon>0$. As for \eqref{ya_conv_3}, an even
stronger statement holds: $Y^\ast_c(u) \to Y^\ast(u)$ as $c \to
\infty$ a.s. for all $u \in \mr$. This can be checked in
the same way as \eqref{as_conv_func} using the fact (see Lemma
\ref{limit_process_properties}(i)) that
\begin{eqnarray*}
\me\bigg[\sum_{k \in \mz} |X_{k}(u+S_k^\ast)| \wedge 1\bigg]  ~<~ \infty.
\end{eqnarray*}
Further, \eqref{bill_cond_4} is a consequence of
$$\lim_{c\to\infty}\limsup_{t\to\infty}\mmp\bigg\{\Big|\sum_{k\geq 0}X_{k+1}(u+t-S_k)\1_{\{|t-S_k| > c\}}\Big|>\varepsilon\bigg\} = 0$$
which has been checked on p.~15 in
\cite{Iksanov+Marynych+Meiners:2015-2} (the independence between
$X$ and $\xi$ was not used in that proof), and equality
\eqref{neg_term} which holds for $t>c>0$. The proof of Theorem
\ref{main1} is complete.

\section{Appendix}
{\bf Proof of Lemma \ref{rmpp_conv}}. Observe that, for all $t\geq
0$,
$$
\sum_{k\in\mz}\delta_{(t-S_k,X_{k+1})}=\sum_{k\in\mz}\delta_{(t-S_{\nu(t)-k-1},X_{\nu(t)-k})}.
$$
Hence, according to Proposition \ref{conv_of_mpp}, it is enough to
prove that, for all $p,q\in\mn$, as $t\to\infty$,
\begin{equation}\label{lemma_rmpp_conv1}
\Big((t-S_{\nu(t)+p-1},X_{\nu(t)+p}),\ldots,(t-S_{\nu(t)-q-1},X_{\nu(t)-q})\Big)\dod
\Big((S^{\ast}_{-p},X_{-p}),\ldots,(S^{\ast}_{q},X_{q})\Big).
\end{equation}
Let $y,z,y_1,\ldots,y_{p-1},z_1,\ldots,z_q$ be
arbitrary nonnegative numbers and $A_{-p},\ldots,A_{q}$ arbitrary
sets in $\mathcal{B}(D(\mr))$. Define the events
\begin{eqnarray*}
\mathcal{E}_{1}(t)&:=&\{t-S_{\nu(t)-1}< z,t-S_{\nu(t)} \geq -y,X_{\nu(t)}\in A_0\},\\
\mathcal{E}_{2}(t)&:=&\{\xi_{\nu(t)+1}\leq y_1,X_{\nu(t)+1}\in A_{-1},\ldots, \xi_{\nu(t)+p-1}\leq y_{p-1},X_{\nu(t)+p-1}\in A_{-p+1}\},\\
\mathcal{E}_{3}(t)&:=&\{\xi_{\nu(t)-1}\leq z_1,X_{\nu(t)-1}\in A_{1},\ldots,\xi_{\nu(t)-q}\leq z_q,X_{\nu(t)-q}\in A_q\},\\
\mathcal{E}_{4}(t)&:=&\{X_{\nu(t)+p}\in A_{-p}\}.
\end{eqnarray*}
Then \eqref{lemma_rmpp_conv1} is equivalent to
\begin{eqnarray*}
\lim_{t\to\infty}\mmp\{\mathcal{E}_{1}(t)\cap\mathcal{E}_{2}(t)\cap\mathcal{E}_{3}(t)\cap\mathcal{E}_{4}(t)\}&=&\mmp\{S_{-1}^{\ast} \geq  -y,S^{\ast}_0 < z,X_0\in A_0\}\\
&&\hspace{-4cm}\times\mmp\{\xi_{-1}\leq y_{1},X_{-1}\in A_{-1},\ldots,\xi_{-p+1}\leq y_{p-1},X_{-p+1}\in A_{-p+1}\}\\
&&\hspace{-4cm}\times\mmp\{\xi_{1}\leq z_{1},X_{1}\in A_{1},\ldots,\xi_{q}\leq z_q,X_{q}\in A_{q}\}\mmp\{X_{-p}\in A_{-p}\}\\
&&\hspace{-4cm}=\mmp\{S_{-1}^{\ast}\geq -y,S^{\ast}_0 < z,X_0\in A_0\}\mmp\{X\in A_{-p}\}\\
&&\hspace{-4cm}\times\prod_{i=1}^{p-1}\mmp\{\xi\leq y_i,X\in A_{-i}\}\prod_{i=1}^{q}\mmp\{\xi\leq z_i,X\in A_i\}.
\end{eqnarray*}
The probability on the left-hand side can be replaced with
$p(t):=\mmp\{\mathcal{E}_{1}(t)\cap\mathcal{E}_{2}(t)\cap\mathcal{E}_{3}(t)\cap\mathcal{E}_{4}(t),\nu(t)>q\}$,
for $\mmp\{\nu(t)\leq q\}\to 0$ as $t\to\infty$.
Conditioning on $\nu(t)$ we obtain
\begin{eqnarray*}
p(t)&=&\sum_{k\geq q+1}\mmp\{\mathcal{E}_{1}(t)\cap\mathcal{E}_{2}(t)\cap\mathcal{E}_{3}(t)\cap\mathcal{E}_{4}(t),\nu(t)=k\}\\
&=&\sum_{k\geq q+1}\mmp\{\mathcal{E}_{1}(t)\cap\mathcal{E}_{2}(t)\cap\mathcal{E}_{3}(t)\cap\mathcal{E}_{4}(t),S_{k-1}\leq t,S_k>t\}\\
&=&\sum_{k\geq q+1}\mmp\{\mathcal{E}_{1}(t)\cap\mathcal{E}_{3}(t),S_{k-1}\leq t,S_k>t\}\\
&&\hspace{4cm}\times\mmp\{X_{k+p}\in
A_{-p}\}\prod_{i=1}^{p-1}\mmp\{\xi_{k+i}\leq y_i,X_{k+i}\in
A_{-i}\}.
\end{eqnarray*}
Thus, it remains to show that
\begin{eqnarray}
&&\lim_{t\to\infty}\sum_{k\geq q+1}\mmp\{\mathcal{E}_{1}(t)\cap\mathcal{E}_{3}(t),S_{k-1}\leq t,S_k>t\}\label{lemma_rmpp_conv2}\\
&&\hspace{4cm}=\mmp\{-S_{-1}^{\ast}\leq y,S^{\ast}_0 < z,X_0\in A_0\}\prod_{i=1}^{q}\mmp\{\xi\leq z_i,X\in A_i\}.\notag
\end{eqnarray}
We continue as follows
\begin{eqnarray*}
&&\hspace{-0.5cm}\mmp\{\mathcal{E}_{1}(t)\cap\mathcal{E}_{3}(t),S_{k-1}\leq t,S_k>t\}\\
&&=\mmp\{t-z < S_{k-1}\leq t,t<S_k\leq t+y,X_k\in A_0,\\
&&\hspace{1cm}\xi_{k-1}\leq z_1,X_{k-1}\in A_1,\ldots,\xi_{k-q}\leq z_q,X_{k-q}\in A_{q}\}\\
&&=\int_{[0,\,t]}\mmp\{t-z-v < \xi_{k-1}+\ldots+\xi_{k-q}\leq t-v,t-v<\xi_{k}+\ldots+\xi_{k-q}\leq t+y-v,\\
&&\hspace{1cm}X_k\in A_0,\xi_{k-1}\leq z_1,X_{k-1}\in A_1,\ldots,\xi_{k-q}\leq z_q,X_{k-q}\in A_{q}\}\mmp\{S_{k-q-1}\in {\rm d}v\}\\
&&=\int_{[0,\,t]}\mmp\{t-z-v < \xi_{1}+\ldots+\xi_{q}\leq t-v,t-v<\xi_{1}+\ldots+\xi_{q}+\hat{\xi}\leq t+y-v,\\
&&\hspace{1cm}\hat{X}\in A_0,\xi_{1}\leq z_1,X_{1}\in A_1,\ldots,\xi_{q}\leq z_q,X_{q}\in A_{q}\}\mmp\{S_{k-q-1}\in {\rm d}v\},\\
&&=:\int_{[0,\,t]}\hat{p}(t-v)\mmp\{S_{k-q-1}\in {\rm d}v\}
\end{eqnarray*}
where
\begin{eqnarray*}
&&\hat{p}(u):=\mmp\{u-z < S_q \leq u,u<S_q+\hat{\xi}\leq u+y,\hat{X}\in A_0,\\
&&\hspace{6cm}\xi_{1}\leq z_1,X_{1}\in A_1,\ldots,\xi_{q}\leq
z_q,X_{q}\in A_{q}\},
\end{eqnarray*}
and $(\hat{X},\hat{\xi})$ has the same distribution as $(X,\xi)$
and is independent of everything else.

The nonnegative function $\hat{p}(u)$ is dRi because it is locally
Riemann integrable and bounded from above by a nonincreasing
integrable function $u\mapsto \mmp\{S_q+\hat{\xi}>u\}$. 
By the key renewal theorem, we infer
\begin{eqnarray*}
\sum_{k\geq q+1}\mmp\{\mathcal{E}_{1}(t)\cap\mathcal{E}_{3}(t),S_{k-1}\leq t,S_k>t\}&=&\sum_{k\geq q+1}\int_{[0,\,t]}\hat{p}(t-v)\mmp\{S_{k-q-1}\in {\rm d}v\}\\
&\to&\frac{1}{\mu}\int_0^{\infty}\hat{p}(u){\rm d}u
\end{eqnarray*}
as $t\to\infty$. Hence proving \eqref{lemma_rmpp_conv2} amounts to
checking the equality
\begin{equation}\label{lemma_rmpp_conv3}
\mmp\{-S_{-1}^{\ast}\leq y,S^{\ast}_0 < z,X_0\in A_0\}\prod_{i=1}^{q}\mmp\{\xi\leq z_i,X\in A_i\}=\frac{1}{\mu}\int_0^{\infty}\hat{p}(u){\rm d}u.
\end{equation}
Set 
$\mathcal{E}:=\{\xi_{1}\leq z_1,X_{1}\in A_1,\ldots,\xi_{q}\leq
z_q,X_{q}\in A_{q}\}$ and rewrite $\hat{p}(u)$ as follows:
\begin{eqnarray*}
\hat{p}(u)&=&\mmp\{u-z\wedge \hat{\xi} < S_q \leq u-(\hat{\xi}-y)_{+},z\wedge \hat{\xi} \geq (\hat{\xi}-y)_{+},\hat{X}\in A_0,\mathcal{E}\}\\
&=&\mmp\{S_q+z\wedge \hat{\xi}>u,z\wedge \hat{\xi} \geq (\hat{\xi}-y)_{+},\hat{X}\in A_0,\mathcal{E}\}\\
&-&\mmp\{S_q+(\hat{\xi}-y)_{+}>u,z\wedge \hat{\xi} \geq (\hat{\xi}-y)_{+},\hat{X}\in A_0,\mathcal{E}\}.
\end{eqnarray*}
Then 
\begin{eqnarray*}
\frac{1}{\mu}\int_0^{\infty}\hat{p}(u){\rm d}u&=&\frac{1}{\mu}\me \Big((S_q+z\wedge \hat{\xi})\1_{\{z\wedge \hat{\xi} \geq (\hat{\xi}-y)_{+},\hat{X}\in A_0,\mathcal{E}\}}\\
&&\hspace{4cm}-(S_q+(\hat{\xi}-y)_{+})\1_{\{z\wedge \hat{\xi} \geq (\hat{\xi}-y)_{+},\hat{X}\in A_0,\mathcal{E}\}}\Big)\\
&=&\frac{1}{\mu}\me\Big(z\wedge \hat{\xi} - (\hat{\xi}-y)_{+}\Big)_{+}\1_{\{\hat{X}\in A_0,\mathcal{E}\}}\\
&=&\frac{1}{\mu}\me\Big(z\wedge \hat{\xi} - (\hat{\xi}-y)_{+}\Big)_{+}\1_{\{\hat{X}\in A_0\}}\mmp\{\mathcal{E}\}\\
&=&\frac{1}{\mu}\me\Big(z\wedge \hat{\xi} - (\hat{\xi}-y)_{+}\Big)_{+}\1_{\{\hat{X}\in A_0\}}\prod_{i=1}^{q}\mmp\{\xi\leq z_i,X\in A_i\},
\end{eqnarray*}
where we have used that $\mathcal{E}$ is independent of
$(\hat{X},\hat{\xi})$. Applying formula
\eqref{zero_interval_joint_distribution} we 
arrive at
\begin{eqnarray*}
\mmp\{-S_{-1}^{\ast}\leq y,S^{\ast}_0 < z,X_0\in A_0\}&=&\mmp\{(1-U)\xi_0\leq y,U\xi_0 < z,X_0\in A_0\}\\
&=&\int_0^{1}\mmp\{\xi_0\leq y(1-s)^{-1},\xi_0 < zs^{-1},X_0\in A_0\}{\rm d}s\\
&\overset{\eqref{zero_interval_joint_distribution}}{=}&\frac{1}{\mu}\int_0^{1}\int_{[0,\,y(1-s)^{-1}\wedge zs^{-1}]}t\mmp\{\hat{\xi}\in {\rm d}t,\hat{X}\in A_0\}{\rm d}s\\
&=&\frac{1}{\mu}\int_{[0,\,\infty)}\Big(\int_{(1-y/t)_{+}}^{z/t\wedge 1} t{\rm d}s\Big)\mmp\{\hat{\xi}\in {\rm d}t,\hat{X}\in A_0\}\\
&=&\frac{1}{\mu}\int_{[0,\,\infty)}\Big(z\wedge t-(t-y)_{+}\Big)_{+}\mmp\{\hat{\xi}\in {\rm d}t,\hat{X}\in A_0\}\\
&=&\frac{1}{\mu}\me\Big(z\wedge \hat{\xi} -
(\hat{\xi}-y)_{+}\Big)_{+}\1_{\{\hat{X}\in A_0\}},
\end{eqnarray*}
and \eqref{lemma_rmpp_conv3} follows. The proof of Lemma
\ref{rmpp_conv} is complete.

\vspace{1cm}
\footnotesize
\noindent   {\bf Acknowledgements}  \quad
The author expresses his sincere gratitude to Prof. Alexander Iksanov for the fruitful discussions and numerous valuable comments on the early draft.
\normalsize

\bibliographystyle{abbrv}
\bibliography{bibliography}

\end{document}